\documentclass[12pt]{amsart}
\usepackage[utf8]{inputenc}

\usepackage{tikz-cd}
\usepackage{verbatim} % 
\usepackage{graphicx}
\usepackage{mathtools}
\usepackage[all]{xy} 
\usepackage{mathrsfs} 
\usepackage{enumerate}
\usepackage{MnSymbol}

\usepackage{tikz-cd}

\DeclareMathAlphabet\mathbb{U}{msb}{m}{n}

\def\ZZ{\mathbb{Z}}
\newcommand{\mono}{\rightarrowtail}
\newcommand{\epi}{\twoheadrightarrow}

\numberwithin{equation}{section}

\newtheorem{Theorem}{Theorem}[section]
\newtheorem{Corollary}[Theorem]{Corollary}
\newtheorem{Lemma}[Theorem]{Lemma}
\newtheorem{Proposition}[Theorem]{Proposition}
\theoremstyle{definition}
\newtheorem{Definition}[Theorem]{Definition}
\newtheorem{Remark}[Theorem]{Remark}

\def\Im{{\rm I}{\sf m}}

\begin{document}

\title{Right exact localizations of groups}

\author{Danil Akhtiamov}
\address{Laboratory of Modern Algebra and Applications,  St. Petersburg State University, 14th Line, 29b,
Saint Petersburg, 199178 Russia} 
\email{akhtyamoff1997@gmail.com}

\author{Sergei O. Ivanov}
\address{Laboratory of Modern Algebra and Applications,  St. Petersburg State University, 14th Line, 29b,
Saint Petersburg, 199178 Russia} 
\email{ivanov.s.o.1986@gmail.com}

\author{Fedor Pavutnitskiy}
\address{Laboratory of Modern Algebra and Applications,  St. Petersburg State University, 14th Line, 29b,
Saint Petersburg, 199178 Russia}
\email{fedor.pavutnitskiy@gmail.com}

\thanks{The main result of the paper (Theorem \ref{theorem_s.d._of_nilpotent}) was obtained under the support
of the Russian Science Foundation grant N 16-11-10073. The authors also are
supported by the grant of the Government of the Russian Federation for the
state support of scientific research carried out under the supervision of leading
scientists, agreement 14.W03.31.0030 dated 15.02.2018.}

%\maketitle

\begin{abstract}
We introduce several classes of localizations (idempotent monads) on the category of groups and study their properties and relations. The most interesting class for us is the class of localizations which coincide with their zero derived functors.  We call them right exact (in the sense of Keune). We prove that a right exact localization $L$ preserves the class of nilpotent groups and that for a finite $p$-group $G$ the map $G\to LG$ is an epimorphism. We also prove that some examples of localizations (Baumslag's $P$-localization with respect to a set of primes $P,$ Bousfield's $HR$-localization, Levine's localization, Levine-Cha's $\ZZ$-localization) are right exact.  At the end of the paper we discuss a conjecture of Farjoun about Nikolov-Segal maps and prove a very special case of this conjecture. 
\end{abstract}
\maketitle
\section*{\bf Introduction}

By a localization we mean an endofunctor $L:\mathcal C\to \mathcal C$ on a category $\mathcal C $ together with a natural transformation $\eta:{\sf Id} \to L$ such that $\eta L=L \eta :L\overset{\cong }\to L^2$ is an isomorphism. An object $c\in \mathcal C$ is called local with respect to a localization $L$ if the morphism $\eta_c: c\to Lc$ is an isomorphism.  Localizations are also known as idempotent monads or idempotent coaugmented functors.
Equivalently one can think about localizations as reflections to reflective subcategories. Historically, functors of this kind appeared in topology, for example Sullivan localization \cite{Sullivan} and Bousfield homological localization \cite{BousfieldLocSpace}. Later, the connection of these localizations with localizations on the category of groups, such as $HR$-localization of \cite{Bousfield77} was discovered. On the other hand, Levin introduced an algebraic closure in \cite{Levine}, now called Levine's localization as a tool to study link concordance. Later it was found that Levine's localization is closely connected to the parafree conjecture \cite{cochranorr} in group theory. One more example of localization comes from the theory of uniquely divisible groups, see \cite{Baumslag1}. This broad variety of examples calls for a systematic study of localization functors on the category of groups. Such study was started in works of Farjoun, Libman  \cite{Libman}, Casacuberta \cite{Casacuberta} and others. The present work is a continuation of the above-mentioned research program. 

%Localizations are interesting from the point of view of category theory  and from the point of view of algebraic topology \cite{Farjoun}, \cite{Adams}, \cite{CasacubertaPeschkePfenniger}. 

We are interested in localizations on the category of groups. There are several works which are focused on properties that are preserved by localizations (see \cite{Casacuberta} for discussion). For example it is known that the following classes of groups are preserved by localizations: abelian groups; nilpotent
groups of class two or less; abelian groups of bounded exponent; finite abelian groups; divisible
abelian groups; rings; commutative rings; fields; modules over rings \cite{Libman}, \cite{Casacuberta}. On the other hand there are counterexamples: the classes of perfect groups  \cite{RodriguezSchererViruel}, finite groups \cite{Libman}; solvable groups, metabelian groups \cite{OSullivan} in general are not preserved by localizations. However, it is not known whether a localization of a nilpotent group is nilpotent. We find this question the most intriguing question in this theory. 

When we work with concrete examples of localizations such as Bousfield's $HR$-localization, or Baumslag's $P$-localization with respect to a set of primes $P$, we see that they satisfy some nice categorical properties which do not hold for arbitrary localizations. This motivates us to consider  some classes of `nice' localizations and ask which properties of groups are preserved by `nice' localizations, and which known examples of localizations are `nice'. 

In this note we introduce several classes of localizations on the category of groups, study relations between them and their properties.  We also prove that some concrete examples are in these classes.  The most interesting class among them is the class of localizations which are equal to their zero derived functor. We call such localizations right exact. In the first two parts of Section \ref{section_right_exact-functors} we discuss various definitions of right exactness for functors ${\sf Gr}\to{\sf Gr}$ and equivalences between these definitions. The third part of Section \ref{section_right_exact-functors} is based on this discussion and devoted to the proof of Theorem \ref{th_s-d-l}, which is our main tool for working with right exact localizations:

\

\noindent {\bf Theorem \ref{th_s-d-l}}. {\it Let $L$ be a localization on the category of groups. Then the following statements are equivalent.
\begin{enumerate}
\item The functor $L$ is right exact.
\item For a local group $A$ and any its normal subgroup $U$, if the semidirect product $A\ltimes U$ is local, then the quotient $A/U$ is local.

\item For any slimplicial group  with local components $G_\bullet$ the group $\pi_0(G_\bullet)$ is local.  
\end{enumerate} 
}

 % More precisely, for any functor $\Phi:{\sf Gr}\to {\sf Gr}$ there is a natural  map from its zero derived functor ${\sf L}_0\Phi\to \Phi$ (see Section \ref{section_right_exact-functors} for details). We say that a functor is right exact, if this map is an isomorphism $ {\sf L}_0\Phi\cong \Phi$. {\bf moreover ...} 

 Further in Section \ref{section-right-exact-nilpotent} we prove the following theorem.

\ 

\noindent {\bf Theorem \ref{theorem_s.d._of_nilpotent}}. {\it Let $L$ be a right exact localization on the category of groups and $G$ be a nilpotent group of class $n.$ Then $LG$ is a nilpotent group of class at most $n.$ Moreover, if $G$ is a finite $p$-group, then $G\to LG$ is an epimorphism. }

%\noindent (see Theorem \ref{theorem_s.d._of_nilpotent}). 
We also prove that some interesting examples of localizations are right exact:

\ 

\noindent {\bf Theorem.} {\it Baumslag's $P$-localization with respect to a set of primes $P,$ Bousfield's $HR$-localization, Levine's localization and Levine-Cha's $\ZZ$-localization are right exact. }

\noindent (see Section \ref{section_examples_s.d} for details).

Further, we consider the following five classes of localizations $L$ on the category of groups.
\begin{itemize}
\item $ \mathfrak{L}_{\sf right\ exact}.$  Right exact  localizations. 

\item $ \mathfrak{L}_{\sf freely\ defined}.$ {\it Freely defined} localizations are localizations $L$ which are localizations with respect to the class of morphisms $W=\{F\to LF\mid F \text{\ is free} \}.$ Roughly speaking, these are localizations which are determined by their  values on free groups.  

\item $\mathfrak{L}_{\sf equational}.$ {\it Equational localizations} are localizations with respect to a class of homomorphisms $\{F_\alpha\to G_\alpha\mid \alpha \in \mathfrak{A}\},$ where $F_\alpha$ is a free group. We also give an equivalent description of this class in the language of systems of equations which explains the term (Proposition \ref{prop_equational_equivalent}).

\item $\mathfrak{L}_{\sf epi}.$ A localization $L$ is called {\it epi-preserving} if for any epimorphism of groups $G\epi H$ the map $LG\to LH$ is also an epimorphism.  We prove that a localization is epi-preserving if and only if the image of a homomorphism between local groups is local (Proposition \ref{prop_epi_equivalent}).

\item $\mathfrak{L}_{\sf all}.$ All localizations. 

\end{itemize}
 
We prove that this is a strictly increasing sequence of classes:
$$\mathfrak{L}_{\sf right\ exact} \subsetneqq \mathfrak{L}_{\sf freely\ defined} \subsetneqq \mathfrak{L}_{\sf equational} \subsetneqq \mathfrak{L}_{\sf epi} \subsetneqq \mathfrak{L}_{\sf all}.$$
In particular, we give concrete examples from the differences of these classes.

In the last section of the paper we discuss the action of Baumslag $P$-localization (see Section \ref{section_examples_s.d} for the definition) on nilpotent groups. In particular we show that in this case the localization map $G\to LG$ is a NS-map: as in \cite{NikolovSegal}, the map $f:H\to G$ is called a NS-map, if $[G,G] = [G,\Im f].$ This answers the conjecture of Farjoun in case of Baumslag $P$-localization of nilpotent groups. 

\subsection*{Acknowledgements} We are grateful to Emmanuel Dror Farjoun for helpful
discussions. We also want to thank Andrei Lavrenov for the information about right exact functors in the sense of Keune.   

\section{\bf Preliminaries}\label{section-preliminaries}
In this section we review some necessary background material on localizations and simplicial homotopy theory. This material will be used in the next section in discussion of right exact functors ${\sf Gr}\to{\sf Gr}.$
\subsection{Basic notions of localizations}

Let us introduce some terminology first. If $L$ is a localization on the category of groups, we say that a group $A$ is local (or $L$-local if we want to emphasize $L$) if $\eta_A:A\to LA$ is an isomorphism. Note that a localization is uniquely defined by the class of local groups because the map $\eta_G: G\to LG$ is the universal map to a local group. 
We say that a map $f:G\to H$ is an equivalence (or $L$-equivalence) if $Lf$ is an isomorphism.  Let $W$  be a class of maps. A group $A$ is said to be $W$-local  if $w^*:{\sf Hom}(H,A)\to {\sf Hom}(G,A)$ is a bijection for any $(w:G\to H)\in W.$ A localization $L$ is called   localization with respect to $W$ if the class of $L$-local groups is the class of $W$-local groups. Note that a localization is always the  localization with respect to the class of its equivalences.  If $W$ is a set, the localization with respect to $W$  exists \cite[Cor. 1.7]{CasacubertaPeschkePfenniger}. Assuming Vop\v{e}nka principle,  all localizations can be presented as localization with respect to one map $W=\{f:G\to H\}$ (see \cite{casacuberta-groups}, also \cite{CasacubertaScevenelsSmith}). The localization with respect to a map $f:G\to H$ is denoted by $L_f.$ For example, Baumslag's $P$-localization (Section \ref{section_examples_s.d}) is the localization with respect to the map $\ZZ\to \ZZ[P^{-1}].$

\subsection{Background on simplicial groups}

% For a given functor on the category of groups sometimes it is necessary to analyse its behaviour on ``higher structures'', such as short exact sequences of groups. Homotopy theory is involved in such analysis, in particular the theory of derived functors. There are several other approaches to the definition of derived functors (for example, see \cite{Inassaridze}, \cite{TierneyVogel}).

In the discussion of homotopical aspects of right exactness we prefer to use the language of simplicial groups. Let us remind it.

We use the standard structure of model category on the category of simplicial groups $s{\sf Gr}$ as a framework for doing homotopy theory in this category, see \cite{GoerssJardine}.  Following Kan \cite{Kan}, \cite{DwyerKan}, we say that a simplicial group $F_\bullet$ is free if all groups $F_n$ are free and their bases are stable under degeneracy maps. In terms of the model category structure, free simplicial groups are cofibrant objects in the category of simplicial groups \cite{DwyerKan}, \cite{Quillen}. Moreover, a simplicial group is cofibrant if and only if it is a retract of a free simplicial group. For any simplicial group $G_\bullet$ there is a cofibrant replacement $F_\bullet\overset{\sim} \epi G_\bullet,$ where $F_\bullet$ is a free simplicial group, and this replacement can be made functorial. For example, one can start with the counit of the adjunction $\mathcal G:{\sf sSets}\rightleftarrows {\sf sGr}:  \overline{W}$ between Kan loop group functor $\mathcal G$ and the classifying space functor $\overline{W}$ (see
 \cite[Th. 3.3]{DwyerKan}, \cite[Ch V. Prop. 6.3]{GoerssJardine}). For a discrete group $G$, considered as a constant simplicial group, we will call such cofibrant replacement $F_\bullet  \overset{\sim}\epi G$ a free resolution of $G,$ if $F_\bullet$ is a free simplicial group. A free resolution is unique up to a strong homotopy equivalence (loop homotopy equivalence) \cite[Prop. 6.5]{Kan}.  We will use this notion later in Section \ref{section_right_exact-functors} to define derive functors of a given functor $\Phi:{\sf Gr}\to {\sf Gr}.$

% A free resolution of a group $G$ is a trivial fibration $F_\bullet  \overset{\sim}\epi G$ from a free simplicial group $F_\bullet$ to the group $G$ considered as a constant simplicial group. 

Recall \cite[Def. 3.6]{curtis-simplicial} that to any simplicial group $K_{\bullet}$ one can assign a chain complex of (in general non-abelian) groups $N K_{\bullet},$ called {\it the Moore complex} of $K_{\bullet},$ defined as
\begin{gather*}
NK_0 = K_0\\
NK_n = \bigcap_{i = 1}^n {\sf Ker}(d_i:K_n\to K_{n-1}) , \ n>0\\
d = d_0|_{NK_n}:NK_n\to NK_{n-1}.
\end{gather*}
Homology groups of this chain complex are isomorphic to homotopy groups of $K_{\bullet},$ see \cite[Th. 3.7]{curtis-simplicial}.

\subsection{(Pre)crossed modules}\label{subsection-precrossed-modules}

Let $G$ be a group. We say that $U$ is a $G$-group, if $U$ is a group together with  a right action of $G$ on $U$ by automorphisms. Morphisms of $G$-groups are homomorphisms preserving the action of $G.$ The category of $G$-groups is denoted by $G$-${\sf Gr}$.   
A normal subgroup $U$ of $G$ will be always considered as a $G$-group with the action by conjugation.

Recall that a precrossed module $\partial$ is a $G$-group morphism $\partial:U\to G,$ where $U$ is a $G$-group. The image of a precrossed module is a normal subgroup of $G$ and the quotient by the image is denoted by $\pi_1(\partial)=G/\Im\, \partial.$  The category of precrossed modules is denoted by ${\sf PCr}$. Moreover, if the following condition
$$
u_1^{\partial (u_2)} = u_1^{u_2}
$$
holds for all $u_1, \, u_2\in U$, the precrossed module $\partial$ is called the crossed module. The theory of crossed and precrossed modules is discussed in detail in \cite{brownhiggins}. Importance of these constructions comes from the fact that crossed modules serve as algebraic models for homotopy 2-types in a similar way as simplicial groups algebraically encode connected topological spaces. Here we sketch one aspect of this equivalence.

It is well known that the category of crossed modules is equivalent to the category of simplicial groups whose Moore complex is of length one (\cite[Lemma 2.2]{Loday}, see also \cite{DonadzeInassaridzePorter} and \cite[Theorem 1.132]{MikhailovPassi}). Moreover, this equivalence of categories preserves homotopy groups. For a crossed module $\partial: U\to G$ the corresponding simplicial group $E_\bullet=E_\bullet(U\to G)$ is called the nerve of the crossed module $\partial .$ It has the following components 
$$E_0=G, \hspace{1cm} E_{n+1}=E_n\ltimes U=(\dots (G\ltimes U)\ltimes \dots )\ltimes U.$$
Its face and degeneracy homomorphisms are defined as
\begin{equation}
\begin{split}
d_0(g,u_1,\dots,u_n)&=(g\partial(u_1),u_2,\dots,u_n), \\
d_i(g,u_1,\dots, u_n)&=(g,u_1,\dots,u_iu_{i+1},\dots,u_n), \hspace{1cm}  0<i<n,\\
d_n(g,u_1,\dots,u_n)&=(g,u_1,\dots,u_{n-1}),\\
s_i(g,u_1,\dots,u_n)&=(g,u_1,\dots,u_{i-1},1,u_i,\dots, u_n), \hspace{1cm} 0\leq i\leq n.
\end{split}
\end{equation}
Then $\pi_0(E_\bullet)=H$ and $\pi_i(E_\bullet)=0$ for $i\ne 0.$

\section{\bf Right exact functors} \label{section_right_exact-functors}

The aim of this section is to remind the notion of a right exact functor in the sense of Keune \cite{Keune}, prove that a functor is right exact if and only if it coincides with its zero derived functor and give some other equivalent descriptions of right exact functors (Theorem \ref{theorem_s.d.f}). 
% In particular, we prove that right exact functors are precisely cosheaves in the sense of Inassaridze \cite{Inassaridze}.  

\subsection{Definitions of right exact functors}

Let $G,H$ be groups. Following Keune \cite{Keune} we denote by double arrow $f:G\Rightarrow H$ a triple of homomorphisms $\alpha=(\alpha_0,\alpha_1,\alpha^0),$ where $\alpha_0, \alpha_1:G\to H$ and $\alpha^0:H\to G$ such that $\alpha_0 \alpha^0={\sf id}_H= \alpha_1 \alpha^0.$ We call such triple a {\it splitted couple of epimorphisms}. 
$$G \overset{\alpha}\Longrightarrow H, 
\hspace{2cm}
\begin{tikzcd}
G \arrow[r,"\alpha_0",shift left=6mm] \arrow[r," \alpha_1",shift left=-5mm] & H \arrow[l,"\alpha^0"']
\end{tikzcd},
\hspace{2cm} 
\alpha_0 \alpha^0={\sf id}_H= \alpha_1 \alpha^0.
$$
Splitted couples of epimorphisms arise in the beginning of the diagram of a simplicial group.  
For a splitted couple of epimorphisms  $\alpha:G\Rightarrow H$ we set 
$$I(\alpha):=\{\ (\alpha_0 g, \alpha_1 g)\in H\times H \mid g\in G\ \}.$$
Note that $I(\alpha)={\rm I}{\sf m}((\alpha_0,\alpha_1): G\to H\times H).$ 
  On the other hand, for a homomorphism $\varphi :G\to H$ Keune defines the following set 
$$K(\varphi):=\{\ (g_0,g_1)\in G\times G \mid \varphi g_0=\varphi g_1\ \},$$
which is a pull-back $G\times_H G.$

\begin{Lemma}\label{lemma_Keune_diagram} Consider a diagram
\begin{equation}\label{eq_Keune_diagram}
G'\overset{\alpha}\Longrightarrow G \overset{\varphi}\longrightarrow G'' 
,
\end{equation}
where $\alpha$ is a splitted couple of epimorphisms and $\varphi$ is an epimorphism. Then the following statements are equivalent. 
\begin{enumerate}
\item There is an equality $I(\alpha)=K(\varphi)$ (in particular $\varphi \alpha_0=\varphi \alpha_1$).
\item If we set $N={\sf Ker}(\alpha_1),$ then the sequence 
$$N \overset{\tilde \alpha_0}\longrightarrow G \overset{\varphi}\longrightarrow G''\longrightarrow 1$$
is exact, where $\tilde{\alpha_0}$ is the restriction of $\alpha_0.$ 
\item The map $\varphi$ is the coequaliser of the couple $(\alpha_0,\alpha_1).$

\item  There is an equality $\varphi \alpha_0=\varphi \alpha_1$ and the map $(\alpha_0,\alpha_1):G'\to G\times_{G''} G$ is an epimorphism.  
\end{enumerate}
\end{Lemma}
\begin{proof}
(1)$\Rightarrow$(2). For any  $n\in N$ we have $ (\alpha_0 n,1)\in I(\alpha)=K(\varphi).$ Then $\varphi(\alpha_0 n)=1.$ Hence $ {\rm I}{\sf m}(\tilde \alpha_0)\subseteq {\sf Ker}(\varphi).$ If $g\in {\sf Ker}(\varphi),$ then $(g,1)\in K(\varphi)=I(\alpha).$ Then there exists $ g'\in G'$ such that $(\alpha_0g',\alpha_1g' )=(g,1).$ Then $g'\in N,$ and $g\in  {\rm I}{\sf m}(\tilde{\alpha}_0).$ Therefore $ {\rm I}{\sf m}(\tilde \alpha_0)= {\sf Ker}(\varphi).$ 

(2)$\Rightarrow$(1). First we prove that $\varphi \alpha_0=\varphi \alpha_1.$ Indeed, for any $g'\in G'$ we have $ g'\cdot (\alpha^0 \alpha_1g')^{-1} \in N.$ Set $n:=g'\cdot (\alpha^0 \alpha_1  g')^{-1}.$ Then, since $\varphi\tilde \alpha_0 = 1,$ we have $1=\varphi \alpha_0 n = \varphi \alpha_0 (g'\cdot (\alpha^0 \alpha_1g')^{-1})= (\varphi\alpha_0g') \cdot (\varphi\alpha_1g')^{-1}.$ Hence $\varphi\alpha_0 g'=\varphi\alpha_1 g'.$

Since $\varphi \alpha_0=\varphi \alpha_1,$ we have $I(\alpha)\subseteq K(\varphi).$ Let $(g_0,g_1)\in K(\varphi).$ Then there exists $n\in N$ such that $\alpha_0 n=g_0\cdot g_1^{-1}.$ Take $g':=n \cdot \alpha^0 g_1.$ Then  $\alpha_0g'=g_0$ and $\alpha_1 g'=g_1.$ Hence $(g_0,g_1)\in I(\alpha).$ Therefore $I(\alpha)=K( \varphi).$

(3)$\Leftrightarrow$(2). We claim that $H:={\rm I}{\sf m}( \tilde \alpha_0 :N\to G)$ is normal in $G.$ Indeed $(\alpha_0 n)^g=\alpha_0(n^{\alpha^0 g}) $ and $n^{\alpha^0 g} \in N.$ Then we need to prove that the projection $\varphi:G\to G/H$ is the coequalizer of the couple of maps $(\alpha_0,\alpha_1).$ Since (1)$\Leftrightarrow $ (2), we obtain that $\varphi \alpha_0=\varphi \alpha_1.$ On the other hand, for any homomorphism $\psi: G\to G''$ satisfying $\psi \alpha_0=\psi \alpha_1 $ we have $ H\subseteq {\sf Ker}( \psi ).$ The assertion follows. 

(1)$\Leftrightarrow$(4). It follows from the equations $K(\varphi)=G\times_{G''}G$ and $I(\alpha)={\rm I}{\sf m}((\alpha_0,\alpha_1):G'\to G\times G).$
\end{proof}

\begin{Definition}[Right exact functor]
Such a diagram \eqref{eq_Keune_diagram} satisfying the equivalent statements (1),(2),(3),(4) of Lemma \ref{lemma_Keune_diagram} is called a {\it right exact diagram}. Following Keune \cite{Keune}, we say that a functor 
$\Phi:{\sf Gr}\to {\sf Gr}$ is {\it right exact (in the sense of Keune)} if it preserves right exact diagrams. Lemma \ref{lemma_Keune_diagram} implies that a functor is right exact if and only if it commutes with coequalisers of splitted couples of epimorphisms. 
\end{Definition}

\begin{Remark} Right exact functors may not commute with coequalizers of arbitrary couples of homomorphisms. They commute only with coequalizers of splitted couples of epimorphisms.   
\end{Remark}

\begin{Remark} 
Some authors define right exact functors as functors commuting with finite colimits. We do not use this definition. By a right exact functor we always mean a right exact functor in the sense of Keune.
\end{Remark}

 Alternatively, as mentioned in Section \ref{section-preliminaries}, one can define derived functors ${\sf L}_n \Phi$ of $\Phi$ using cofibrant replacement functor: for $G\in {\sf Gr}$ define 
\begin{equation}\label{eq-derived-functors}
{\sf L}_n \Phi (G):=\pi_n (\Phi F_\bullet),
\end{equation}
where $F_\bullet$ is a free simplicial resolution of $G$. The extension of $\Phi$ to the category of simplicial groups sends a strong homotopy equivalence to a strong homotopy equivalence in $s{\sf Gr},$ hence the definition does not depend on the choice of $F_\bullet$. In this paper we are only interested in the zero derived functor ${\sf L}_0\Phi$. It is easy to see that there is a natural map ${\sf L}_0\Phi\to \Phi.$ Below we prove (Theorem \ref{theorem_s.d.f}) that functor is right exact if and only if the map is an isomorphism ${\sf L}_0\Phi\cong \Phi.$

As an example of right exact functor, one can consider a cartesian square $\Phi: G\mapsto G\times G$ with diagonal map as a coaugmentation. Indeed, it is easy to see that its derived functors ${\sf L}_n\Phi$ are trivial for $n>0$ and ${\sf L}_0\Phi = \Phi.$

Let $G$ be a group. Let us discuss one more construction $\Phi_G$ derived from a given functor $\Phi$, now in connection with the category of $G$-groups. Let $\Phi:{\sf Gr}\to {\sf Gr}$ be a functor. For a group $G$ we consider a functor from the category of $G$-groups to the category of $\Phi G$-groups
$$ \Phi_G :G\text{-}{\sf Gr} \longrightarrow \Phi G\text{-}{\sf Gr} $$
given by 
$$\Phi_GU:={\sf Ker}(\Phi(G\ltimes U) \to \Phi G).$$
The action of $\Phi G$ on $\Phi_G U$ is determined by the map $\Phi G\to \Phi(G\ltimes U).$ Note that there is an isomorphism 
$$\Phi(G\ltimes U)=\Phi G\ltimes \Phi_GU.$$

Consider a precrossed module $\partial: U\to G.$ It induces a homomorphism 
$$\mu_\partial: G\ltimes U \longrightarrow G, \hspace{1cm} \mu_\partial(g,u)=g\partial(u).$$
This induces a map
$$ \mu^{\Phi}_\partial: \Phi_GU \longrightarrow \Phi G,$$
which is the composition of $\Phi(\mu_\partial)$ and the embedding $\Phi_G U \hookrightarrow \Phi(G\ltimes U).$ It is easy to check that $ \mu^{\Phi}_\partial$ is a precrossed module again, where the action of $\Phi G$ on $\Phi_G U$ goes via the map $\Phi G \to \Phi (G\ltimes U).$ This gives a functor
$$\Phi^{\sf PCr}:{\sf PCr}\longrightarrow {\sf PCr}, \hspace{1cm} \Phi^{\sf PCr}(\partial)=\mu_\partial^\Phi.$$

In case when $U$ is normal subgroup of $G$ we set $\mu_U:=\mu_\partial$ and 
$\mu^\Phi_U:=\mu^\Phi_\partial.$

\subsection{Equivalence of definitions}
 Now we will show that definitions of right exactness in terms of diagrams \eqref{eq_Keune_diagram} and in terms of derived functors \eqref{eq-derived-functors} are equivalent.

\begin{Theorem}[{cf. \cite[Th.8]{Keune},  \cite[Prop. 2.26]{Inassaridze}}] \label{theorem_s.d.f} Let $\Phi:{\sf Gr} \to {\sf Gr} $ be a functor. Then the following statements are equivalent. 
\begin{enumerate}
\item The functor $\Phi$ is right exact.

\item The natural map ${\sf L}_0 \Phi \overset{\cong}\to \Phi$ is an isomorphism.
\item For a short exact sequence $U\mono G \epi H$ the sequence
$$ \Phi_G U \overset{ \mu_U^\Phi}\longrightarrow \Phi G \longrightarrow \Phi H \longrightarrow 1 $$
is exact. 
\item For a precrossed module $\partial$ the  natural morphism $\pi_1(\Phi^{\sf PCr} (\partial) )\overset{\cong}\to \Phi(\pi_1(\partial))$ is an isomorphism. 
\item For a simplicial group $G_\bullet$ the natural morphism $\pi_0(\Phi G_\bullet)\overset{\cong}\to \Phi (\pi_0 G_\bullet)$ is an isomorphism.

\item For any epimorphism of groups $G\epi H$ the following diagram is a coequalizer diagram
$$\Phi( G\times_H G) \rightrightarrows \Phi(G)\to \Phi(H),$$
where the maps that are coequalized are natural projections $G\times_H G\to G.$
\end{enumerate}
\end{Theorem}
\begin{proof} 
First we prove that (2),(3),(4),(5) are equivalent.

(2)$\Rightarrow$(3). Since $G\epi H$ is an epimorphism we can choose free resolutions $F_\bullet \epi G$ and $F_\bullet' \epi H$ and the lifting $F_\bullet \to F_\bullet'$ of the epimorphism such that $F_n\to F'_n$ is an epimorphism for any $n$. For example, we can choose the bar resolutions. An epimorphism between free groups splits. It follows that $\Phi(F_0)\to \Phi(F_0')$ is a split epimorphism. Therefore ${\sf L}_0\Phi G \to {\sf L}_0 \Phi H$ is an epimorphism, and hence, $\Phi G\epi \Phi H$ is an epimorphism. So, $\Phi$ is epi-preserving.

Now consider a crossed module $\partial: U\hookrightarrow G$ and its classifying space $E_\bullet$ as in Section \ref{subsection-precrossed-modules}. Take a cofibrant replacement by a free simplicial group $F_\bullet \overset{\sim}\epi E_\bullet.$
Then $F_\bullet$ is a free resolution of $H.$ We obtain a diagram
\begin{equation}\label{diagram-map-of-resolutions}
\begin{tikzcd}
\Phi F_1\arrow[r,shift left = 1mm] \arrow[r,shift left = -1mm] \arrow[d,twoheadrightarrow] & \Phi F_0 \arrow[d,twoheadrightarrow]  
\\
\Phi(G\ltimes U)\arrow[r,shift left = 1mm] \arrow[r,shift left = -1mm] & \Phi G 
\end{tikzcd}
\end{equation}
 If we set $N={\sf Ker}(d_1:F_1\to F_0),$ the map $s_0:F_0\to F_1$ induces an isomorphism $F_1\cong F_0\ltimes N.$ Note that $d_1:G\ltimes U \to G$ is just the projection but $d_0=\mu_U:G\ltimes U\to G.$

 After switching to Moore complexes, from \eqref{diagram-map-of-resolutions} we obtain the following diagram
$$ 
\begin{tikzcd}
\Phi_{F_0} N\arrow[r] \arrow[d,twoheadrightarrow] & \Phi F_0 \arrow[r,twoheadrightarrow] \arrow[d,twoheadrightarrow] & {\sf L}_0\Phi H\arrow[d,"\cong"]  \\
\Phi_G U\arrow[r," \mu_U^\Phi"] & \Phi G \arrow[r,twoheadrightarrow] & \Phi H,
\end{tikzcd}
$$
where the upper row is exact, vertical maps are epimorphisms and the right-hand vertical map is an isomorphism. Simple diagram chasing shows that the lower row is also exact.

(3)$\Rightarrow$(4). Take a precrossed module $\partial: U\to G.$ Set $I:=\Im \partial$ and $H=\pi_1 \partial.$ Note that   (3) implies that $\Phi$ is epi-preserving. It follows that $\Phi_G U \to \Phi_G I$ is surjective. Then we obtain the diagram
$$
\begin{tikzcd}
\Phi_G U\arrow[r,"\mu^\Phi_\partial"]\arrow[d,twoheadrightarrow] & \Phi G\arrow[r,twoheadrightarrow]\arrow[d,equal] & \pi_1 (\Phi^{\sf PCr}(\partial))\arrow[d] \\
 \Phi_G I\arrow[r,"\mu^\Phi_I"] & \Phi G\arrow[r,twoheadrightarrow] & \Phi H.
\end{tikzcd}
$$  
The upper row is exact by definition. The lower row is exact by assumption. Then the five-lemma implies that the right hand map is an isomorphism.

(4)$\Rightarrow$(5). For any simplicial group $G_\bullet$ one can construct a precrossed module $N_{\leq 1}(G_\bullet)=(\partial:N_1\to G_0),$ where $N_1={\sf Ker}(d_1:G_1\to G_0)$ and $\partial=d_0|_{N_1}.$ This gives a functor $N_{\leq 1}:{\sf sGr}\to {\sf PCr}$ such that $\pi_1( N_{\leq 1}(G_\bullet))=\pi_0(G_\bullet).$ It is easy to check that the diagram 
$$
\begin{tikzcd}
{\sf sGr}\arrow[r,"N_{\leq 1}"]\arrow[d,"\Phi"] & {\sf PCr}\arrow[d,"\Phi^{\sf PCr}"] \\
{\sf sGr}\arrow[r,"N_{\leq 1}"] & {\sf PCr}
\end{tikzcd}
$$ 
is commutative. The assertion follows. 

(5)$\Rightarrow$(2). Obvious.

So we proved that (2),(3),(4),(5) are equivalent. 

(1)$\Leftrightarrow$(4). The languages of splitted couples of epimorphisms and precrossed modules are equivalent. Let us give translations.  For a splitted  couple of epimorphisms $\alpha: G\Rightarrow H $ we denote by ${\sf N}(\alpha)$ the precrossed module 
$ \tilde \alpha_0: N\to H,$ where $\tilde \alpha_0$ is the restriction of $\alpha_0$ and $H$ acts on $N$ by the formula $n^h=n^{\alpha^0 h}.$ It is easy to see that the group $G$ can be reconstructed from the precrossed module because there is an isomorphism $H\ltimes N\cong G, (h,n)\mapsto \alpha^0h\cdot n.$
On the other hand, if we have a precrossed module $\partial: N\to H$ we denote by ${\sf C}(\partial)$ the splitted couple of epimorphisms $\alpha :H\ltimes N \Rightarrow H$, where 
$$\alpha_0(h,n)=h\cdot n, \hspace{0.5cm} \alpha_1(h,n)=h,  \hspace{0.5cm}\alpha^0(h)=(h,1).$$ It is easy to check that these two functors ${\sf N}:{\sf SplCoupEpi}\to {\sf PCr}$ and ${\sf C}:  {\sf PCr} \to {\sf SplCoupEpi}$ induce equivalences of categories of splitted couples of epimorphisms and precrossed modules 
$${\sf SplCoupEpi} \simeq {\sf PCr}.$$ 
There is a commutative diagram of functors
$$
\begin{tikzcd}
{\sf SplCoupEpi} \arrow[d,"\Phi"] \arrow[r,"{\sf N}"] & {\sf PCr}\arrow[d,"\Phi^{\sf PCr}"] \\
{\sf SplCoupEpi} \arrow[r,"{\sf N}"]  & {\sf PCr}
\end{tikzcd}.
$$   
Coequalizer of a splitted couple of epimorphisms $\alpha$ is isomorphic to $\pi_1({\sf N}(\alpha))$ by Lemma \ref{lemma_Keune_diagram}. The assertion follows. 

Thus we proved that (1),(2),(3),(4),(5) are equivalent. 

(1)$\Rightarrow$(6). The triple $\alpha=({\sf pr}_0, {\sf pr}_1, \Delta),$ where ${\sf pr}_0,{\sf pr}_1: G\times_H G\to G$ are projections and $\Delta:G\to G\times_H G$ is a diagonal map, forms a splitted couple of epimorphisms $\alpha:G\times_H G \Rightarrow G$. The assertion follows. 

(6)$\Rightarrow$(1). (6) implies that $\Phi$ sends epimorphisms to epimorphisms. Let $G'\Rightarrow G \to G''$ be a right exact diagram. Lemma \ref{lemma_Keune_diagram} implies that $G'\to G\times_{G''} G$ is an epimorphism. Hence $\Phi(G') \to \Phi(G\times_{G''} G)$ is an epimorphism. It follows that the coequalizer of $\Phi(G')\rightrightarrows \Phi(G)$ equals to the coequalizer of $\Phi(G\times_{G''} G)\rightrightarrows \Phi(G).$ The assertion follows. 
\end{proof}

\subsection{Right exact localizations and equivariantly local groups}

%\section{\bf Right exact localizations and equivariantly local groups}

In this part of the paper we introduce a convenient reformulation of the property of localization functor being right exact in terms of locality of quotients of local groups by local normal subgroup of special kind (Theorem \ref{th_s-d-l}), which will be our main tool for working with right exact localizations.

Let $L=(L,\eta)$ be a localization on the category of groups and $A$ be a local group with respect to $L.$ If we identify $A=LA,$ we obtain that the functor $L_A$ acts from the category of $A$-groups to itself 
$$L_A: A\text{-}{\sf Gr}\longrightarrow A\text{-}{\sf Gr}$$
such that 
$$L(A\ltimes U)=A\ltimes L_AU.$$
Moreover, there is a natural map $U\to L_AU$ and $L_A$ is a localization on the category of $A$-groups. We say that an $A$-group $U$ is equivariantly local, if it is local with respect to $L_A.$ Note that $U$ is equivariantly local if and only if $A\ltimes U$ is local.

\begin{Lemma}\label{lemma_ker_local}
If $A$ and $B$ are local groups and $f:A\to B$ is a homomorphism, then ${\sf Ker} f$ is equivariantly local. 
\end{Lemma}
\begin{proof}
The class of local groups is closed under small limits. Thus $A\times_B A$ is also local. The map $A\times_B A\to A\ltimes {\sf Ker} f$ given by $(a_1,a_2)\mapsto (a_1,a_1^{-1}a_2)$ is an isomorphism. Hence, $A\ltimes {\sf Ker} f $ is local. Therefore, ${\sf Ker} f$ is equivariantly local. 
\end{proof}

\begin{Corollary} 
If $A$ is a local group, then $A$ is equivariantly local as $A$-group. 
\end{Corollary}

\begin{Corollary} 
If $A$ is a  local group and $U\triangleleft A$  such that $A/U$ is local, then $U$ is equivariantly local.
\end{Corollary}

\begin{Theorem}\label{th_s-d-l} Let $L$ be a localization on the category of groups. Then the following statements are equivalent.
\begin{enumerate}
\item The functor $L$ is right exact.
\item For a local group $A$ and any equivariantly local normal subgroup $U$ the quotient $A/U$ is local.  

\item For any slimplicial group  with local components $G_\bullet$ the group $\pi_0(G_\bullet)$ is local.
\end{enumerate} 
\end{Theorem}
\begin{proof}
(1)$\Rightarrow$ (2). Follow from (2)$\Rightarrow$(3) of Theorem \ref{theorem_s.d.f}. 

(2)$\Rightarrow$ (1).  
For any group $G$ we have $L_{LG} \circ L_G=L_G$ because there are isomorphisms $$LG\ltimes L_G U = L(G\ltimes U)=LL(G\ltimes U)=L(LG\ltimes L_G U)=LG\ltimes L_{LG}L_G U,$$
which are compatible with maps to $LG.$
It follows that $L_GU$ is $LG$-local for any $G$-group $U.$

Lemma \ref{lemma_ker_local} together with the assumption implies that the image of a homomorphism $A\to B$ between two local groups is local. It follows that if $U$ and $V$ are two equivariantly local $A$-groups and $U\to V$ is an $A$-group morphism, then its image is also equivariantly local. 

Consider a short exact sequence $U\mono G\epi H.$  Then the groups $LG$ and $L_GU$ are $LG$-local, and hence, the image $I=\Im (L_G U \to LG)$ is a $LG$-local normal subgroup. Therefore $LG/I$ is local by assumption. Hence $LG/I$ is the cokernel of the map $L_GU\to LG$ in the category of local groups. On the other hand the functor $L$ considered as a functor to the category of local groups $L':{\sf Gr}\to {\sf Loc}$ commutes with cokernels because it is left adjoint to the embedding. Then $LG/I\cong LH.$ The assertion follows from (3)$\Rightarrow$(2) of Theorem \ref{theorem_s.d.f}. 

(1)$\Rightarrow$(3). This follows from (1)$\Rightarrow$ (5) of Theorem \ref{theorem_s.d.f}.

(3)$\Rightarrow$(2). Let $A$ be a local group and $U$ be its equivariantly local normal subgroup. Consider the nerve $E_\bullet=E_\bullet(U\hookrightarrow A)$ of the crossed module $U\hookrightarrow A.$ Set $G_\bullet:=LE_\bullet.$ Since $A$ is local and $U$ is equivarianly local, we have $G_0=E_0=A$ and $G_1=E_1=A\ltimes U.$ Therefore $\pi_0(G_\bullet)=\pi_0(E_\bullet)=A/U,$ and hence, $A/U$ is local.
\end{proof}

\begin{Corollary}\label{cor_image_of_local} If $L$ is a right exact localization and $A$ and $ B$ are local groups, then the image of a homomorphism $A\to B$ is local.   
\end{Corollary}
\begin{proof}
It follows from Lemma \ref{lemma_ker_local} and Theorem \ref{th_s-d-l}. 
\end{proof}
\begin{Remark}
From Theorem \ref{th_s-d-l} it is clear that if for the localization $L$ the category of $L$-local groups forms a variety (in sense of universal algebra), then localization $L$ is right exact. Indeed, since variety is closed under epimorphic images, every quotient of $L$-local group is $L$-local. In particular, nilpotent reduction functor $G\mapsto G/\gamma_n$ is an example of right exact localization.
\end{Remark}

\section{\bf Right exact localizations of nilpotent groups}\label{section-right-exact-nilpotent}

In the following section we establish a behaviour of right exact localizations with respect to central extensions. Using the properties of this behaviour we are able to answer some questions of \cite{Casacuberta} for a class of right exact localizations.

We denote by ${\sf Z}(G)$ the center of a group $G.$ We say that an epimorphism $f: G\epi H$ is a  central extension if ${\sf Ker}(f)\subseteq {\sf Z}(G).$

\begin{Lemma}\label{lemma_image_center}
Let $L$ be a localization on the category of groups and $f:G\to H$ be a homomorphism. Then  ${\rm I}{\sf m} (f) \subseteq {\sf Z}(H) $ implies $\Im (Lf) \subseteq {\sf Z}(LH).$
\end{Lemma}
\begin{proof}
We claim that that $\eta({\sf Z}(H))\subseteq {\sf Z}(LH),$ where $\eta=\eta_H:H\to LH$ is the coaugmentation.  Indeed for any $z\in {\sf Z} (H)$  we obtain that the inner automorphism $\varphi_{\eta z}:LH\to LH$ satisfies $\varphi_{\eta z} \eta = \eta,$ and hence, $\varphi_{\eta z}={\sf id}$ and $\eta z\in {\sf Z}(LH).$ It follows that $\Im (\eta_H f)=\Im((Lf) \eta_G)\subseteq {\sf Z}(LH).$ Take an element 
$b\in LH$ and denote by  $ \varphi_b:LH\to LH$ the corresponding inner automorphism. Since $\Im (\eta_Hf)\subseteq {\sf Z}(LH),$ we obtain $\varphi_b  (Lf) \eta_G=(Lf) \eta_G.$ Since $LH$ is local, we have that $\eta_G^*:{\sf Hom}(LG,LH)\cong {\sf Hom}(G,LH),$ and hence $\varphi_b  (Lf) =Lf$ for any $b\in LH.$ The assertion follows. 
\end{proof}

\begin{Lemma}[c.f. \cite{DwyerFarjoun}]\label{lemma_central_extension} Let $L$ be a right exact localization on the category of groups and $G\epi H$ be a central extension. Then $LG\epi LH $ is also a central extension and, if we set $K={\sf Ker}(G\epi H),$ the sequence $LK\to LG \to LH \to 1$ is exact.
\end{Lemma}
\begin{proof} We have an isomorphism $L_GK=LK$ because $G\ltimes K=G\times K$ and $L$ commutes with products. Hence the sequence $LK\to LG \to LH \to 1$ is exact by (1)$\Rightarrow$(2) of Theorem \ref{theorem_s.d.f}. Lemma \ref{lemma_image_center} implies that the image of $LK\to LG$ is in the center.    
\end{proof}

\begin{Theorem}\label{theorem_s.d._of_nilpotent}
Let $L$ be a right exact localization on the category of groups and $G$ be a nilpotent group of class $n.$ Then $LG$ is a nilpotent group of class at most $n.$ Moreover, if $G$ is a finite $p$-group, then $G\to LG$ is an epimorphism. 
\end{Theorem}
\begin{proof}
It is known that for an abelian group $G$ its localization $LG$ is also abelian \cite{Libman} (it also  follows from Lemma \ref{lemma_image_center}). Then the result about nilpotent groups follows by induction from Lemma \ref{lemma_central_extension}. 

Now consider the second assertion. For any localization $L$ we have $L(\ZZ/p)\cong \ZZ/p$ or $L(\ZZ/p)=0$ (see \cite[Th.2.3]{Libman}). Here we use the fact that a localization on the category of groups induces a localization on the category of abelian groups. This gives the basis of induction. For the step of induction it is enough to prove that for a central extension $G\epi H$ such that ${\sf Ker}(G\epi H)\cong \ZZ/p,$ 
if $H\to LH$ is an epimorphism, then $G\to LG$ is an epimorphism. Denote by $I$ the image of $L(\ZZ/p)\to LG.$  Then Lemma \ref{lemma_central_extension} implies that $I\mono LG\epi LH$ is a short exact sequence. 
Therefore 
we obtain a morphism of short exact sequences
$$ 
\begin{tikzcd}
\ZZ/p\arrow[r,rightarrowtail]\arrow[d,twoheadrightarrow]  & G \arrow[r, twoheadrightarrow] \arrow[d] & H \arrow[d,twoheadrightarrow] \\
I\arrow[r,rightarrowtail] & LG \arrow[r, twoheadrightarrow] & LH.
\end{tikzcd}
$$ 
The snake lemma implies that the map $G\to LG$ is an epimorphism.  
\end{proof}

\section{\bf Examples of right exact localizations} \label{section_examples_s.d}

\subsection{Baumslag's $P$-localization}

Let $P$ be a set of prime numbers. We denote by $\ZZ[P^{-1}]$ the subring of $\mathbb Q,$ consisting of numbers whose denominator's prime factors are in $P.$ We denote by $L_{P}$ the localization functor with respect to the map $\ZZ\to \ZZ[P^{-1}].$  A group $A$ is local with respect to $L_{P}$ if and only if it is uniquely $P$-divisible i.e. the $p$-power map $A\to A,a\mapsto a^p$ is a bijection for any $p\in P$ (such groups are called $D$-groups in Baumslag's papers \cite{Baumslag1}, \cite{Baumslag2}). Note that on the subcategory of abelian groups Baumslag $P$-localization is just tensoring with $\ZZ[P^{-1}].$

\begin{Theorem}\label{th_Baumslag_l}
Baumslag's localization at a set of primes $P$ is right exact. 
\end{Theorem}
\begin{proof}
We are going to prove this statement using our characterising property of right exact localizations given in Theorem \ref{th_s-d-l}. Let us consider a local group $A$ and its equivariantly local normal subgroup $U$. Then $A$ and $A\ltimes U$ are uniquely $P$-divisible. It is easy to check by induction that  $(a,u)^n=(a^n, a^{-n}(au)^n)$ for any $(a,u) \in A\ltimes U.$ Therefore,  $A\ltimes U$ is uniquely $P$-divisible if and only if for any $p\in P$ and any $(a,u)\in A\ltimes U$ there exists a unique $(b,v)\in A\ltimes U$ such that  $(b^p, b^{-p}(bv)^p)=(a,u)$. Since $A$ is uniquely $P$-divisible, this is equivalent to the following statement: 
\begin{center}
$(*)$ for any $p\in P,$ $a\in A$ and $u\in U$ there exists a unique $v\in U$ such that $u=a^{-p}(av)^p$.
\end{center}
Thus it is sufficient to show that $A/U$ is uniquely $P$-divisible. It is easy to see that it is $P$-divisible, hence only uniqueness need to be proved. We prove that for any $a_1,a_2\in A$ such that  $a_1^p a_2^{-p} \in U $ we have $a_1a_2^{-1}\in U.$
Take $u=a_1^pa_2^{-p}$  and $a=a_1^{-1}$. Thus $(*)$ implies that there exists  $v \in U$ such that $a_1^pa_2^{-p}=a_1^p(a_1^{-1}v)^p$. 
Dividing both of the sides by $a_1^p$ and extracting the $p$-th root, we obtain: $a_2^{-1}=a_1^{-1}v$, and hence, $a_1a_2^{-1}=v \in U$. It follows that $A/U$ is uniquely $P$-divisible.
\end{proof}

\subsection{Bousfield's $HR$-localization}

We denote by $R$ a subring of $\mathbb Q$ or the ring $\ZZ/n$ for some $n.$ Then the functor of $HR$-localization is the functor of localization with respect to the class of $R$-$2$-connected homomorphisms i.e. homomorphisms $f:G\to H$ which induce an isomorphism $H_1(G,R)\cong H_1(H,R)$ and an epimorphism $H_2(G,R)\epi H_2(H,R)$ (for details see \cite{Bousfield77}). Nilpotent groups are examples of $H\ZZ$-local groups, but in general computation of $H\ZZ$-localization of the given group can be a difficult problem. Example of such computation is given in \cite[Prop. 9.2]{IvanovMikhailov} for $G=\ZZ\ltimes\ZZ.$

Following Bousfield \cite{Bousfield77}  we say that a subgroup $B$ of a group $A$ is {\it $HR$-closed}, if for any bigger subgroup  $B\subsetneq B'\subseteq A$ the map $H_1(B,R)\to H_1(B',R)$ is not an epimorphism.

Bousfield has proved the following statements. 

\begin{Proposition}[Bousfield \cite{Bousfield77}]\label{prop_bous} \ 
\begin{enumerate}
\item A subgroup of an $HR$-local group is $HR$-local if and only if it is $HR$-closed \cite[Cor. 2.11]{Bousfield77}. 

\item The image of a homomorphism between $HR$-local groups is $HR$-local \cite[Cor. 2.12]{Bousfield77}.

\item A homomorphism between $HR$-local groups $A\to B$ is an epimorphism if and only if $H_1(A,R)\to H_1(B,R)$ is an epimorphism \cite[Cor.2.13]{Bousfield77}.

\item In particular, the functor of $HR$-localization is epi-preserving. 
\end{enumerate}
\end{Proposition}

\begin{Definition} Let $A$ be a group and let $U$ be an $A$-group. We say that 
an $A$-subgroup $V\leq U$ is {\it $HR\text{-}A$-closed}, if $A\ltimes V$ is $HR$-closed in $A\ltimes U.$ Note that $H_1(A\ltimes V,R)=H_1(A,R)\oplus H_1(V,R)_A.$ Hence, $V$ is $HR\text{-}A$-closed in $U$ if and only if for any bigger $A$-subgroup $V\subsetneq V'\subseteq U$ the map between coinvariants $H_1(V,R)_A\to H_1(V',R)_A$ is not an epimorphism. 
\end{Definition}

\begin{Lemma}\label{lemma_local_qoutient} Let $A$ be an $HR$-local group and let $U$ be a normal subgroup. Then the following statements are equivalent. 
\begin{enumerate}
\item The quotient $A/U$ is $HR$-local.
\item The quotient $A/U$ is transfinitely $R$-nilpotent, i.e. intersection of the transfinite lower $R$-central series, as in \cite[p. 48]{Bousfield77}, is trivial.
\item The group $U$ is $HR$-$A$-closed in $A$.   
\item The group $U$ is $L_A$-local, where $L$ is the functor of $HR$-localization.
\end{enumerate}
\end{Lemma}
\begin{proof}
(1)$\Rightarrow$(2). Any $HR$-local group is transfinitely $R$-nilpotent \cite[Prop. 1.2]{Bousfield77}. 

(2)$\Rightarrow$(3). Assume that $U\subseteq  V\triangleleft A$ and $H_1(U,R)_A\to H_1(V,R)_A$ is an epimorphism. It follows that $U/[U,A]\otimes R\to V/[V,A]\otimes R$ is an epimorphism. If we set $V'=V/U$ and $A'=A/U,$ we obtain $V'/[V',A']\otimes R=0.$ Since $A/U $ is transfinitely $R$-nilpotent, this implies $V'=1,$ and hence $V=U.$

(3)$\Rightarrow$(1).
 We denote the functor of $HR$-localization by $L$ and $H_*(-):=H_*(-,R).$ 
Since $L$ sends epimorphisms to epimorphisms, we obtain that $A\to L(A/U)$ is an epimorphism.  Denote the kernel of this epimorphism by $U'$ and consider the diagram:
$$
\begin{tikzcd}
1 \arrow[r] & U \arrow[r]\arrow[d,hookrightarrow] & A \arrow[r]\arrow[d,"{\sf id}"] & A/U \arrow[r]\arrow[d,"\eta_{A/U}"] &1 \\
1\arrow[r]  & U'\arrow[r] & A \arrow[r] & L(A/U)\arrow[r] & 1.
\end{tikzcd}
$$
Since $\eta_{A/U}$ is $R$-2-connected, we obtain the isomorphism $H_1(A/U)\cong H_1(L(A/U))$ and the epimorphism $H_2(A/U)\twoheadrightarrow H_2(L(A/U)).$ Then 5-lemma applied to the diagram
$$
\begin{tikzcd}
H_2(A)\arrow[r] \arrow[d,"\cong"]& H_2(A/U)\arrow[r]\arrow[d,twoheadrightarrow] & H_1(U)_A\arrow[r]\arrow[d] & H_1(A)\arrow[r]\arrow[d,"\cong"] & H_1(A/U)\arrow[d,"\cong"]\\
H_2(A)\arrow[r] & H_2(A/U')\arrow[r] & H_1(U')_A\arrow[r] & H_1(A) \arrow[r]& H_1(L(A/U)).
\end{tikzcd}
$$
implies that $H_1(U)_A\to H_1(U')_A$ is an epimorphism. 
Using that $U$ is $HR\text{-}A$-closed, we obtain that $U=U',$ and hence $A/U=L(A/U)$ is $HR$-local. 

(3)$\Leftrightarrow$ (4).  Proposition \ref{prop_bous} states that a subgroup of an $L$-local group is $L$-local if and only if it is $HR$-closed. Therefore an $A$-subgroup of an $L_A$-local $A$-group is $L_A$-local if and only if it is $HR$-$A$-closed. 
\end{proof}

\begin{Theorem}\label{Bousfield-right-exact}
The functor of $HR$-localization is right exact. 
\end{Theorem}
\begin{proof}
%This follows from Theorem \ref{th_s-d-l} and Lemma \ref{lemma_local_qoutient}.

By Theorem \ref{th_s-d-l} it is sufficient to show that for any $HR$-local group $A$ and its equivariantly local normal subgroup $U$ the quotient $A/U$ is $HR$-local. But this follows Lemma \ref{lemma_local_qoutient} part $(4)\Rightarrow (1).$

\end{proof}

\subsection{ Levine's localization and  Levine-Cha's $\ZZ$-localization}

In this subsection we will discuss Levine's localization \cite{Levine} which is called the  ``algebraic closure'' in that paper (see also \cite{LevineI}, \cite{LevineII}) and show that this localization is right exact (Theorem \ref{Levine-right-exact}). We will also use some facts about the so-called ``$R$-algebraic closure'' of groups introduced by Cha in \cite{Cha}, which we call Levine-Cha $R$-localization. 

For a subgroup $H$ of a group $G$ we denote by $H^G$ the normal subgroup of $G$ generated by $H.$

Let $G$ be a group, $X$ be a set and $F(X)$ be a free group generated by $X$. Elements of $G*F(X)$ are called {\it monomials} (with coefficients in $G$).
A monomial is called {\it contractible} if it lies in the kernel of the projection $G*F(X)\to F(X).$ The set of contractible monomials forms the normal closure of $G$ in $G*F(X).$  

A monomial is called {\it acyclic} if it is in the kernel of the projection to the abelianization $G*F(X)\to F(X)_{ab}.$ The set of acyclic monomials forms the group $G^{G*F(X)}[G*F(X),G*F(X)].$

\begin{Lemma} Let $G$ be a group, $H$ be a subgroup and $\varphi:X\to G$ be a map  such that $G=\langle H\cup \varphi (X)\rangle$. Then the following statements about an element $g\in G$ are equivalent.
\begin{enumerate}
\item The element $g$ lies in $H^G.$

\item There exists a contractible monomial $w\in H*F(X),$ whose image in $G$ is $g.$
\end{enumerate}
Moreover, the following statements about an element $g\in G$ are also equivalent.

\begin{enumerate}
\item The element $g$ lies in $H^G[G,G].$

\item There exists an acyclic monomial $w\in H*F(X),$ whose image in $G$ is $g.$
\end{enumerate}
\end{Lemma}
\begin{proof}
Obvious.
\end{proof}

Let $G$ be a group and $X=\{x_1,\dots,x_n\}$ be a finite set. {\it Levine's system} of equations is a set of the form $E=\{x_1w_1^{-1},\dots,x_nw_n^{-1}\}\subseteq G*F(X),$ where $w_i$ are contractible monomials. One thinks about this set as a set of equations $\{x_i=w_i\}.$ Following Cha \cite{Cha}, we define a {\it $\ZZ$-nullhomologous system of equations} as a set of the form $E=\{x_1w_1^{-1},\dots,x_nw_n^{-1}\}\subseteq G*F(X),$ where $w_i$ are acyclic monomials. 
 A {\it solution} of a Levine's (or $\ZZ$-nullhomologous) system of equations is a map $ \varphi :X\to G$ such that $E$ lies in the kernel of corresponding homomorphism $G*F(X)\to G.$ 

%A group $G$ is called {\it Levine local} if any Levine's system of equations has a unique solution. Levine's localization of a group $G$ is the universal map from $G$ to a Levine local group.  It exists for any group and gives a functor $L_{\sf Lev}:{\sf Gr}\to {\sf Gr}$.
%Levine localization is called the  ``algebraic closure'' and Levine local groups are called ``$E$-local'' in \cite{Levine}. It is known that Levine's localization is the localization with respect to the class of $2$-connected normally surjective homomorphisms $f:G\to H$ such that $G$ is finitely generated and $H$ is finitely presented  \cite{Levine}. 
%
%A group $G$ is called {\it  Levine-Cha $\ZZ$-local} if any $\ZZ$-nullhomologous system of equations has a unique solution. Levine-Cha's $\ZZ$-localization of a group $G$ is the universal map from $G$ to a Levine-Cha $\ZZ$-local group.  It exists for any group and gives a functor $L_{\sf Lev\text{-}Cha}:{\sf Gr}\to {\sf Gr}$ (see \cite{Cha}). The Levine-Cha $\ZZ$-localization is the localization with respect to the class of $2$-connected  homomorphisms $f:G\to H$ such that $G$ is finitely generated and $H$ is finitely presented \cite[Th. 5.2.]{Cha}. 

A group $G$ is called {\it Levine local} (resp. {\it Levine-Cha $\ZZ$-local}) if any Levine's (resp. $\ZZ$-nullhomologous) system of equations has a unique solution. Levine's localization (resp. Levine-Cha's $\ZZ$-localization) of a group $G$ is the universal map from $G$ to a Levine local (resp. Levine-Cha $\ZZ$-local) group. Both localizations exist for any group and give functors $L_{\sf Lev},\ L_{\sf Lev\text{-}Cha}:{\sf Gr}\to {\sf Gr}$ (see \cite{Cha}). Levine localization is called the  ``algebraic closure'' and Levine local groups are called ``$E$-local'' in \cite{Levine}. It is known that Levine's localization is the localization with respect to the class of $2$-connected normally surjective homomorphisms $f:G\to H$ such that $G$ is finitely generated and $H$ is finitely presented  \cite{Levine}. Similarly, the Levine-Cha $\ZZ$-localization is the localization with respect to the class of $2$-connected  homomorphisms $f:G\to H$ such that $G$ is finitely generated and $H$ is finitely presented \cite[Th. 5.2.]{Cha}. As with $H\ZZ$-localization, nilpotent groups are both Levine local and Levine-Cha $\ZZ$-local.

A normal subgroup $H$ of a group $G$ is called {\it invisible} if it is the normal closure of a finite number of elements $H=\langle h_1,\dots, h_n \rangle^G$ and $H=[H,G].$ A Levine local group does not contain  nontrivial invisible subgroups \cite{Levine}. 

\begin{Proposition}[{c.f. Lemma 4.9. \cite{Cha}}] \label{prop_invisible} The following statements about a group $G$ are equivalent. 

\begin{enumerate}
\item The group $G$ has no nontrivial invisible subgroups. 

\item A $\ZZ$-nullhomologous system over $G$ has at most one solution. 

\item A Levine's system over $G$ has at most one solution. 
\end{enumerate}
\end{Proposition}
\begin{proof}

(1)$\Rightarrow$(2). Assume that $\{x_1w_1^{-1},\dots,x_nw_n^{-1}\}$ 
is a $\ZZ$-nullhomologous system of equations and $\varphi, \psi:X\to G$ are two of its solutions. Set $\varphi_i=\varphi(x_i), \psi_i=\psi(x_i)$ and $h_i= \varphi_i\psi^{-1}_i.$ 
We prove that $H:=\langle h_1,\dots, h_n \rangle^G$
 is an invisible subgroup. In the following computation the symbol $\equiv$
  means equation modulo $[H,G].$ Note that $h_ig\equiv gh_i$ for any $g\in G.$  
We need to  
prove that $h_i \equiv 1.$ Fix $i$ and let 
$$w_i=g_1 x_{k_1}^{\epsilon_1} g_2 \dots g_m x_{k_m}^{\epsilon_m} g_{m+1},$$
where $g_j\in G$ and $\epsilon_i\in \{-1,1\}.$ Then 
$$\varphi_i=g_1 \varphi_{k_1}^{\epsilon_1} g_2 \dots g_m \varphi_{k_m}^{\epsilon_m} g_{m+1}$$
and 
$$\psi_i=g_1 \psi_{k_1}^{\epsilon_1} g_2 \dots g_m \psi_{k_m}^{\epsilon_m} g_{m+1}.$$ 
 Set $\varphi_i^{(l)} =g_1 \varphi_{k_1}^{\epsilon_1} g_2 \dots g_{l-1} \varphi_{k_{l-1}}^{\epsilon_l} g_{l}$ and $\psi_i^{(l)}=g_1 \psi_{k_1}^{\epsilon_1} g_2 \dots  g_{l-1} \psi_{k_{l-1}}^{\epsilon_{l-1}} g_{l}.$
 Hence
$$h_i=\varphi_i \psi_i^{-1}=\varphi_i^{(m)} \varphi^{\epsilon_m}_{k_m} g_{m+1} g^{-1}_{m+1} \psi_{k_m}^{-\epsilon_m} (\psi_i^{(m)})^{-1} \equiv \varphi_i^{(m)} (\psi_{i}^{(m)})^{-1} h_{k_m}^{\epsilon_m} \equiv \dots $$
$$ \dots \equiv \varphi_i^{(m-1)} (\psi_{i}^{(m-1)})^{-1} h_{k_{m-1}}^{\epsilon_{m-1}} h_{k_m}^{\epsilon_m} \equiv \dots \equiv h_{k_1}^{\epsilon_1} \dots h_{k_m}^{\epsilon_m}. $$
Since  $x_{k_1}^{\epsilon_1}\dots x_{k_m}^{\epsilon_m}\in [F(X),F(X)],$  we obtain   $h_{k_1}^{\epsilon_1} \dots h_{k_m}^{\epsilon_m}\equiv 1.$ Therefore $h_i\equiv 1.$

(2)$\Rightarrow$(3). Obvious. 

(3)$\Rightarrow$(1). Let $H=\langle h_1,\dots,h_n \rangle^G$ be an invisible subgroup. Then $H=[H,G],$ and hence, $h_i$ can be presented as a product of commutators $h_i=[a_{i,1} ,g_{i,1}]^{\epsilon_{i,1}} { \dots} [a_{i,N},g_{i,N}]^{\epsilon_{i,N}},$ where $a_{i,j}\in H$ and $\epsilon_{i,j}\in \{-1,1\}.$ Set $X=\{x_1,\dots,x_n\}.$ Since $a_{i,j}\in \langle h_1,\dots,h_n \rangle^G$ there exist  monomials $u_{i,j}\in G*F(x_1,\dots,x_n)$ which lie in the normal closure $\langle x_1,\dots,x_n \rangle^{G*F(X)}$ such that $a_{i,j}$ is the image of $u_{i,j}$ with respect to  the map $\{x_1,\dots,x_n\}\to G$ sending $x_i$ to $h_i.$ Consider the monomials $w_i=[u_{i,1} ,g_{i,1}]^{\epsilon_{i,1}} { \dots} [u_{i,N},g_{i,N}]^{\epsilon_{i,N}}$ in $G*F(x_1,\dots,x_n).$ Note that $w_i\in \langle x_1,\dots,x_n \rangle^{G*F(X)}$ because $u_{i,j}$ is in this subgroup.  The map $X\to G$ sending $x_i$ to $h_i$ is a solution of the Levine's system of equations $\{ x_1w_1^{-1},\dots,x_nw_n^{-1} \}$. The map $X\to G$ sending $x_i$ to $1$ is also a solution of this system. Hence $h_i=1$ for any $i$ and $H=1.$
\end{proof}

\begin{Proposition} \label{prop_local_subgroups_Levine} Let $A$ be a Levine local group  and $B$ be its subgroup. Then the following statements are equivalent.
\begin{enumerate}
\item The subgroup $B$ is Levine local.

\item If 
$C=\langle B \cup \{c_1,\dots,c_n\} \rangle$ for some  $c_1,\dots,c_n\in A$  such that 
 $C=B^C,$ then $B=C.$
\end{enumerate}
\end{Proposition}

\begin{proof}
(1)$\Rightarrow$(2).  Since $C$ is the normal closure of $B$ in $C,$ $c_i$ can be presented as $w_i(c_1,\dots,c_n),$ where $w_i$
 is a contractible monomial of $B*F(x_1,\dots,x_n).$  Then $c_i$ is the solution of the Levine's system of equations $\{x_1w_1^{-1},\dots,x_nw_n^{-1}\}$ over $A,$ which is unique. Since $B$ is local, $c_i\in B.$

(2)$\Rightarrow$(1). Assume that $\{x_1w_1^{-1},\dots,x_nw_n^{-1}\}$ is a Levine's system of equations and $c_1,\dots,c_n$ is its solution in $A.$ Thus $C=\langle B\cup \{c_1,\dots, c_n\} \rangle$ satisfies the assumption $C=B^C$, and hence $C=B.$ It follows that the elements of the solution $c_1,\dots,c_n$ lie in $B.$ Then  $B$ is local. 
\end{proof}

\begin{Lemma} Let $U\subseteq V$ be two normal subgroups of a group $G.$ Then
\begin{equation}\label{lem-two-subgroups}
 (G\ltimes U)^{G\ltimes V}= G\ltimes U[V,G]. 
 \end{equation}
 
\end{Lemma}
\begin{proof}
Since commutator subgroup of $G\ltimes V$ is equal to $[G,G]\ltimes [V,G],$ it contains in the right hand side of \eqref{lem-two-subgroups}, hence $G\ltimes U[V,G]$ is normal and contains the normal closure of $G\ltimes U.$ For the inclusion in the opposite direction it is sufficient to show that $G\ltimes [V,G]$ is contained in the left hand side of \eqref{lem-two-subgroups}, this follows from the chain of inclusions
$$
1\ltimes [V,G]\subseteq [G\ltimes V,G\ltimes U]\subseteq  (G\ltimes U)^{G\ltimes V}.
$$
\end{proof}

\begin{Proposition}\label{prop_local_normal_subgroups_Levine} Let $A$ be a Levine local group and $U$ be a normal subgroup. Then the following statements are equivalent.
\begin{enumerate}
\item The quotient $A/U$ is Levine local.

\item The quotient $A/U$ has no nontrivial invisible subgroups. 

\item The subgroup $U$ is equivariantly Levine local.

\item If  $V= \langle U \cup \{v_1,\dots , v_n\} \rangle^A$ for some $v_1,\dots,v_n\in A$  such that $V=U[V,A]$, then $V=U.$   
\end{enumerate}
\end{Proposition}
\begin{proof}
(3)$\Leftrightarrow$(4).
Any subgroup of $A\ltimes A$ containing $A\ltimes U$ is a group of the form $A\ltimes V,$ where $V$ is normal. The semidirect product $A\ltimes V$ is the normal closure of $A\ltimes U$ in  $A\ltimes V$ if and only if $V=U[V,A].$  Then it follows from  Proposition  \ref{prop_local_subgroups_Levine}.

(2) $\Leftrightarrow$ (4). The equality $V=U[V,A]$  is equivalent to the equality $V'=[V',A'],$ where $V'=V/U$ and $A'=A/U.$ Then a normal subgroup $V\supseteq U$ satisfies the assumption of (4) if and only if $V/U$ is invisible. 

(1) $\Leftrightarrow$ (2). It is easy to see that any Levine's system has a (non necessarily unique)  solution in $A/U$ because $A$ is Levine local. Then this follows from Proposition \ref{prop_invisible}. 
\end{proof}

\begin{Theorem}\label{Levine-right-exact} Levine's localization is right exact. 
\end{Theorem}
\begin{proof}
This follows from (1)$\Leftrightarrow$ (3) of  Proposition  \ref{prop_local_normal_subgroups_Levine} and Theorem \ref{th_s-d-l}. 
\end{proof}

Similarly one can prove the following results.

\begin{Proposition}  Let $A$ be a Levine-Cha $\ZZ$-local group and $B$ be its subgroup. Then the following statements are equivalent.
\begin{enumerate}
\item The subgroup $B$ is Levine-Cha $\ZZ$-local.

\item If 
$C=\langle B \cup \{c_1,\dots,c_n\} \rangle$ for some  $c_1,\dots,c_n\in A$  such that 
 $C=B^C[C,C],$ then $B=C.$
\end{enumerate}
\end{Proposition}

\begin{Proposition} Let $A$ be a Levine-Cha $\ZZ$-local group and $U$ be a normal subgroup. Then the following statements are equivalent.
\begin{enumerate}
\item The quotient $A/U$ is Levine-Cha $\ZZ$-local.

\item The quotient $A/U$ has no nontrivial invisible subgroups. 

\item The subgroup $U$ is equivariant Levine-Cha $\ZZ$-local.

\item If  $V= \langle U \cup \{v_1,\dots , v_n\} \rangle^A$ for some $v_1,\dots,v_n\in A$  such that $V=U[V,A]$, then $V=U.$   
\end{enumerate}
\end{Proposition}

\begin{Theorem}\label{Levine-Cha-right-exact}
 Levine-Cha's $\ZZ$-localization is right exact. 
\end{Theorem}

We believe that similar results can be proved for Levine-Cha's $R$-localization for arbitrary $R$.

\section{\bf Freely defined localizations}
In the following three sections we introduce more classes of localizations (with some examples and non-examples) which properly extend the class of right exact localizations.
\subsection{Freely defined localizations} A localization $L$ is called {\it freely defined} if it is a localization with respect to the class of maps $\{\eta_F: F\to LF\},$ where $F$ runs over all free groups. In other words, a group $A$ is local if and only if ${\sf Hom}(LF,A)\to {\sf Hom}(F,A)$ is a bijection for any free group $F$.   

\begin{Proposition} 
Any right exact localization is freely defined, i.e.
$$\mathfrak{L}_{\sf right exact}\subseteq \mathfrak{L}_{\sf freely\ defined}.$$
\end{Proposition}
\begin{proof}
Assume that $L$ is right exact and a group $A$ satisfies  the property that ${\sf Hom}(LF,A)\to {\sf Hom}(F,A)$ is a bijection for any free group $F$.  We prove that $A$ is local. It is enough to prove that ${\sf Hom}(LG,A)\to {\sf Hom}(G,A)$ is a bijection for any group $G.$ Consider a free simplicial resolution $F_\bullet\overset{\sim}\epi G.$ Since $L$ is right exact, $LG$ is the coequaliser of $LF_1 \rightrightarrows LF_0.$ Then we have a commutative diagram
$$
\begin{tikzcd}
{\sf Hom}(LG,A)\arrow[r]\arrow[d] & {\sf Hom}(LF_0,A)\arrow[r,shift left=1mm]\arrow[r,shift left=-1mm]\arrow[d,"\cong"] & {\sf Hom}(LF_1,A)\arrow[d,"\cong"]\\ 
 {\sf Hom}(G,A)\arrow[r] & {\sf Hom}(F_0,A)\arrow[r,shift left=1mm]\arrow[r,shift left=-1mm] & {\sf Hom}(F_1,A),
\end{tikzcd}
$$ 
where rows are equalizers. Then the left vertical map is also a bijection.  
\end{proof}

\subsection{The localization of the abelianization with respect to the map $\ZZ\to \ZZ_p$}
This subsection is devoted to an example of a freely defined localization, which is not right exact. We denote by $p$ a prime number and by $\ZZ_p$ the abelian group of $p$-adic integers. Note that $${\sf End}(\ZZ_p)\cong \ZZ_p$$ and any endomomorphism of $\ZZ_p$ is given by multiplication by an element of $\ZZ_p.$ Indeed ${\sf End}(\ZZ_p)={\sf Hom}(\ZZ_p, \varprojlim \ZZ/p^i )= \varprojlim {\sf Hom}(\ZZ_p,  \ZZ/p^i )= \varprojlim \ZZ/p^i =\ZZ_p.$  Hence the map $\ZZ\to \ZZ_p$ is local (i.e. ${\sf Hom}(\ZZ_p,\ZZ_p)\cong {\sf Hom}(\ZZ,\ZZ_p)$).

Consider the localization on the category of abelian groups with respect to the map $\ZZ\to \ZZ_p$
$$\ell_{p} : {\sf Ab}\longrightarrow {\sf Ab},\hspace{1cm} \ell_p:=\ell_{\ZZ\to \ZZ_p}.$$
Any localization on the category of abelian groups is additive, so $\ell_{p}$ is also additive and we can consider its derived functors in the sense of classical homological algebra.   

\begin{Lemma}\label{lemma_abelian_localization}  For any free abelian group $A$ we have $\ell_pA=A\otimes \ZZ_p.$ Moreover, the zero derived functor is given by tensoring by $\ZZ_p$ 
$${\sf L}_0\ell_{p}= - \otimes \ZZ_p,$$
which is not a localization functor. 
In particular, $\ell_p$ is not equal to its zero derived functor. 
\end{Lemma}
\begin{proof}
Let $\alpha$ be any cardinal number. We prove that $\ZZ_p^{\oplus \alpha}$ is local with respect to the homomorphism $\ZZ\to \ZZ_p.$ First we note that the product $\ZZ_p^{\prod \alpha}$ is obviously local 
$${\sf Hom}(\ZZ_p,\ZZ_p^{\prod \alpha})\cong {\sf Hom}(\ZZ_p,\ZZ_p)^{\prod \alpha}\cong {\sf Hom}(\ZZ,\ZZ_p)^{\prod \alpha} \cong {\sf Hom}(\ZZ,\ZZ_p^{\prod \alpha}).$$ 
Then we have a commutative diagram
$$
\begin{tikzcd}
{\sf Hom}(\ZZ_p,\ZZ_p^{\oplus \alpha})\arrow[r]\arrow[d,rightarrowtail] & {\sf Hom}(\ZZ,\ZZ_p^{\oplus \alpha})\arrow[d,rightarrowtail] \\
{\sf Hom}(\ZZ_p,\ZZ_p^{\prod \alpha})\arrow[r,"\cong"] & {\sf Hom}(\ZZ,\ZZ_p^{\prod \alpha}).
\end{tikzcd}
$$
This implies that the upper horizontal map is a monomorphism. On the other hand it is an epimorphism because for any map $f:\ZZ\to \ZZ_p^{\oplus \alpha}$ we can lift it using the $\ZZ_p$-module structure $f':\ZZ_p\to \ZZ_p^{\oplus \alpha}$ $f'(a)=a\cdot f(1).$ Thus $ \ZZ_p^{\oplus \alpha}$ is local.

Assume that $B$ is a local group. Then we have 
$${\sf Hom}(\ZZ^{\oplus \alpha},B)\cong {\sf Hom}(\ZZ,B)^{\prod \alpha}\cong {\sf Hom}(\ZZ_p,B)^{\prod \alpha} \cong {\sf Hom}(\ZZ^{\oplus \alpha}_p,B).$$ It follows that the map $\ZZ^{\oplus \alpha} \to \ZZ^{\oplus \alpha}_p$ is a universal map to local groups and $\ell_p(\ZZ^{\oplus \alpha})\cong \ZZ^{\oplus \alpha}_p.$ In other words, for any free abelian group $A$ we have $\ell_p A\cong A\otimes \ZZ_p.$  By the universal property of the morphism $A\to A\otimes \ZZ_p$ we see that the isomorphism $\ell_p A\cong A\otimes \ZZ_p$ is natural. If two additive functors coincide on free abelian groups, then their derived functors coincide. Hence ${\sf L}_0\ell_p=-\otimes \ZZ_p.$ The functor $-\otimes \ZZ_p$ is not a localization because $\ZZ\otimes \ZZ_p\otimes \ZZ_p=\ZZ_p\otimes \ZZ_p \not\cong \ZZ_p=\ZZ\otimes \ZZ_p.$
\end{proof}

\begin{Proposition} Let $L:{\sf Gr}\to {\sf Gr}$ denote the composition  given by $$LG:=\ell_{p}(G_{ab}).$$ Then the following holds.
\begin{enumerate}
\item The functor $L$ is the localization with respect to the map $f:F_2\to LF_2 = \ZZ_p^2,$ where $F_2$ is the 2-generated free group 
$$L=L_f.$$
\item The zero derived functor of $L$ can be described as follows
${\sf L}_0L(G)=G_{ab}\otimes \ZZ_p,$ which is not a localization.
\item The functor $L$ 
is freely defined but it is not right exact, therefore we have:
$$\mathfrak{L}_{\sf right exact}\ne \mathfrak{L}_{\sf freely\ defined}.$$
\end{enumerate}
\end{Proposition}
\begin{proof} It is easy to see that $L$ is a localization. 
Note that $L$-local groups are abelian groups which are local with respect to the map $\ZZ\to \ZZ_p.$ We claim that a group is local if and only if it is local with respect to the map $f:F_2\to \ZZ_p^2.$ Indeed, any $f$-local group $A$ is abelian because for any two elements $a,b\in A$ we can consider a map $F_2\to A$ sending generators to $a,b$ and lift it to a map from an abelian group $\ZZ_p^2\to A.$ On the other hand the map $\ZZ\to \ZZ_p$ is a retract of the map $f,$ and hence, all $f$-local groups are $\ZZ\to \ZZ_p$-local. Hence $L=L_f.$

Consider the set $W=\{F\to LF\mid F\text{\ is free}\}.$ Since $f\in W$ and $L=L_f,$ we obtain that $L$ is freely defined.

Now we prove that ${\sf L}_0L(G)=G_{ab}\otimes \ZZ_p.$   Let us consider a free resolution $F_\bullet\overset{\sim }\epi G$ of a group $G.$ The functor of abelianization is right exact $\pi_0((F_\bullet)_{ab})=G_{ab}.$ Hence $d_0-d_1: (F_1)_{ab}\to (F_0)_{ab}$ is the beginning of a free resolution of $G_{ab}$ in the category of free groups. Then Lemma \ref{lemma_abelian_localization} implies that the cokernel of $\ell_p((F_1)_{ab}) \to \ell_p((F_0)_{ab}) $ is $G_{ab}\otimes \ZZ_p.$ It follows that ${\sf L}_0 L(G)=G_{ab}\otimes \ZZ_p.$ The functor $G\mapsto G_{ab}\otimes \ZZ_p$ is not a localization because $\ZZ_p\not\cong \ZZ_p\otimes \ZZ_p.$
\end{proof}

\section{\bf Equational localizations}

A localization $L$ on the category of groups is called {\it equational} if there exists a class of homomorphisms $W=\{F_\alpha\to G_\alpha \mid \alpha\in \mathfrak{A}\}$ such that $F_\alpha$ is a free group for any $\alpha$ and $L$ is the localization with respect to this class.  It follows from the definition that a freely defined localization is equational, namely,
$$\mathfrak{L}_{\sf freely\ defined}\subseteq \mathfrak{L}_{\sf equational}.$$ 

Let $A,X$ be two sets. Denote by $FA$ and $FX$ the free groups generated by them and by $FA*FX$ their free product. We say that $A$ is the set of parameters and $X$ is the set of variables. A  {\it system of equations} is a triple of sets  $(A,X,E),$ where $E\subseteq FA*FX.$  We say that a group $G$ {\it satisfies} the system of equations $(A,X,E)$ if for any function $A\to G$ there exists a unique function $X\to G$ such that $E$  is contained  in the kernel of the induced map $FA*FX\to G.$ 

For example, consider $A=\{a\},$  $X=\{x\}$ and $E=\{x^2a \}.$ Then $G$ is uniquely $2$-divisible if and only if $G$ satisfies the system of equations $(A,X,E).$

Assume  $\mathcal{E} =\{(A_\alpha,X_\alpha,E_\alpha )\mid \alpha\in \mathfrak{A}\}$ is a class of systems of equations. A group $G$ {\it satisfies}  $\mathcal{E}$ if it satisfies each of the systems. We say that a localization $L$ is {\it defined by the class of equations} $\mathcal{E}$ if the class of $L$-local groups is the class of groups satisfying $\mathcal{E}.$

For example, Bousfield's $H \mathbb Z$-localization can be defined by the class of so-called $\Gamma$-systems of equations \cite{FarjounOrrShelah}. Levine's localization is by definition a localization defined by a class of systems of equations \cite{LevineI}, \cite{LevineII}. Baumslag's $P$-localization is also defined by a class of systems of equations $\mathcal E=\{(\{a\},\{x\},\{x^pa\})\mid p\in P\}.$

\begin{Proposition}\label{prop_equational_equivalent}
A localization is equational if and only if it is defined by some class of systems of equations.  
\end{Proposition}
\begin{proof}
Let $\mathcal{E} =\{(A_\alpha,X_\alpha,E_\alpha )\mid \alpha\in \mathfrak{A}\}$ be a system of equations. Set $$F_\alpha=FA_\alpha, \hspace{1cm} G_\alpha=(FA_\alpha* FX_\alpha) /  \langle \!\langle E_\alpha \rangle \! \rangle,$$
where  $\langle \!\langle E_\alpha \rangle \! \rangle$ denotes the normal subgroup generated by $E_\alpha.$ It is easy to see that a group satisfies $(A_\alpha,X_\alpha,E_\alpha)$ if and only if it is local with respect to the map $F_\alpha\to G_\alpha.$  
%So, if a localization is defined by a system of equations, then it is equational. 
On the other hand, since any map $F\to G$ can be presented in form $FA\to (FA* FX)/E$  for some sets $A,\ X$ and $E,$ any equational localization is defined by some class of systems of equations.  
\end{proof}

\begin{Proposition}
The localization with respect to the map $1\to \ZZ/2$ is equational but it is not freely defined, therefore:
$$\mathfrak{L}_{\sf freely\ defined}\ne \mathfrak{L}_{\sf equational}.$$
\end{Proposition}
\begin{proof}
The localization is equational by definition because the trivial group is free. Moreover, the same localization can be defined as a localization with respect to the homomorphism $F\to F\times \ZZ/2,$ where $F$ is any free group. Indeed, the class of local groups in both cases is the class of $2$-torsion-free groups. Free groups are local with respect to this localization i.e. $LF=F.$ Hence the corresponding class of morphisms $\{F\to LF\}$ consists of identity morphisms. Therefore the localization is not freely defined.  
\end{proof}

\section{\bf Epi-preserving localizations} 

A localization $L$ is called {\it epi-preserving} if for any epimorphism $G\epi H$ the map $LG\to LH$ is also an epimorphism. 

\begin{Proposition}\label{prop_epi_equivalent}
A localization is epi-preserving if and only if the image of a homomorphism between a pair of local groups is local.  
\end{Proposition}
\begin{proof}
Assume that $L$ is epi-preserving $A$ and $B$ are local group and $A\to B$ is a homomorphism. Denote by $I$ the image of this homomorphism. Applying the localization to the epimorphism $A\epi I$ we obtain an epimorphism $A\epi LI.$ It follows that $\eta_I:I\epi LI$ is also an epimorphism. On the other hand, the composition $I\to LI \to B$ is the embedding $I\hookrightarrow B$. Thus  $\eta_I $ is an isomorphism and $I$ is local. 

Assume now that $L$ satisfies the property that an image between local groups is local. Take an epimorphism $G\epi H$ and prove that the map $LG\to LH$ is an epimorphism. Denote by $I$ the image of $LG\to LH$. It is local by the assumption.  Note that the image of $\Im(H\to LH)\subseteq I.$ Indeed $\Im(H\to LH)=\Im(G\epi H\to LH)=\Im(G\to LG\to LH)\subseteq \Im(LG\to LH).$ Then \cite[Lemma 1.7]{Libman} implies that $I=LH.$ 
\end{proof}

\begin{Proposition} An equational localization is epi-preserving, hence
$$\mathfrak{L}_{\sf equational } \subseteq \mathfrak{L}_{\sf epi}. $$
\end{Proposition}
\begin{proof}
Assume that $L$ is defined by a class of maps $\{F_\alpha\to G_\alpha\}.$ Let $A$ and $B$ be local groups and let  $A\to B$ be a homomorphism. By Proposition \ref{prop_epi_equivalent} it is enough to prove that its image $I=\Im(A\to B)$ is local. Take a homomorphism $f:F_\alpha \to I.$ Since $F_\alpha$ is free, we can lift it to a homomorphism $\tilde f:F_\alpha\to A.$ Since $A$ is local, there exists a lifting $\tilde f':G_\alpha \to A $ of $\tilde f.$ Its composition with the projection $A\epi I$ is a  lifting $f':G_\alpha \to I$ of $f:F_\alpha\to I.$ It is unique because $B$ is local. 
\end{proof}

\begin{Proposition} The localization $L_{\ZZ/4\epi \ZZ/2}$ is epi-preserving but it is not equational, therefore we have:
$$ \mathfrak{L}_{\sf equational}\ne \mathfrak{L}_{\sf epi}.$$
\end{Proposition}
\begin{proof}
A group is local with respect to $L_{\ZZ/4\epi \ZZ/2}$ if and only if there is no $4$-torsion in the group (which is not $2$-torsion). Thus a subgroup of a local group is local. Hence by Proposition \ref{prop_epi_equivalent} $L_{\ZZ/4\epi \ZZ/2}$ is epi-preserving. 

Prove that the localization is not equational. Assume the contrary, that $L_{\ZZ/4\epi \ZZ/2}$ is the localization with respect to a class of maps  $\{F_\alpha \to G_\alpha\},$ where $F_\alpha$ is free for any $\alpha$. Let us concentrate our attention on abelian groups. An abelian group is local with respect to the class $\{F_\alpha \to G_\alpha\}$ if and only if it is local with respect to the class of their abelianisations $\{A_\alpha\to B_\alpha\},$ where $A_\alpha=(F_\alpha)_{ab}$ and $B_\alpha=(G_\alpha)_{ab}.$ Since $A_\alpha$ is a free abelian group, it has no $4$-torsion, and hence it is local. Hence $A_\alpha$ is local with respect to the map $A_\alpha\to B_\alpha.$ Therefore $A_\alpha\to B_\alpha$ is a split monomorphism and $B_\alpha=A_\alpha\oplus C_\alpha.$ It follows that an abelian group $D$ is local if and only if ${\sf Hom}(C_\alpha,D)=0$ for any $\alpha.$ On the other hand we have an exact sequence  $$0\longrightarrow {\sf Hom}(C_\alpha,\ZZ/2) \longrightarrow {\sf Hom}(C_\alpha,\ZZ/4) \longrightarrow {\sf Hom}(C_\alpha,\ZZ/2),$$ so if $\ZZ/2$ is local, then $\ZZ/4$ is also local. It is a contradiction because $\ZZ/2$ is local but $\ZZ/4$ is not local.  
\end{proof}

Let $A_n$ denote the alternating group. Libman proved that for $n\geq 7$ the embedding $A_n\hookrightarrow A_{n+1}$ is a local map i.e. $L_{A_n\hookrightarrow A_{n+1}}(A_n)\cong  A_{n+1}$ \cite[Example 3.4]{Libman}. Note that if group $G$ does not contain $A_n$ as a subgroup, it is $L_{A_n\hookrightarrow A_{n+1}}$-local.

\begin{Proposition} The localization $L_{A_n\hookrightarrow A_{n+1}}$ for $n\geq 7$ is not epi-preserving, hence:
$$\mathfrak{L}_{\sf epi}\ne \mathfrak{L}_{\sf all}.$$ 
\end{Proposition}
\begin{proof}
Consider an epimorphism $F\epi A_n$ from a free group and set $L=L_{A_n\hookrightarrow A_{n+1}}$. Note that $F$ is local. Then we obtain a commutative diagram
$$
\begin{tikzcd}
F\arrow[r,twoheadrightarrow] \arrow[d,"\cong"] & A_n \arrow[d,hookrightarrow]  \\
LF\arrow[r] & LA_n = A_{n+1}.
\end{tikzcd}
$$  
Hence the image of $LF\to LA_n=A_{n+1}$ is $A_n,$ so it is not an epimorphism. 
\end{proof}

\section{\bf Nikolov--Segal maps and $P$-localization}

Now we discuss one more conjecture of Emmanuel Dror Farjoun. A homomorphism $f:H\to G$ is called {\it NS-map} (Nikolov--Segal map, see \cite{NikolovSegal}) if it satisfies 
$$[G,G]=[G, \Im f].$$ 
Note that the cokernel (quotient by the normal closure of the image) of an NS-map is abelian. 
Nikolov and Segal proved a theorem \cite[Theorem 1.6]{NikolovSegal} that implies that for a finitely generated group $G$ the map to its pro-finite completion $G\to \hat G_{\sf fin}$ (see \cite[Ex. 2.1.6]{ribes-zalesskii}) is an NS-map.   

\ 

\noindent{\bf Conjecture} (Farjoun). {\it For any localization $L$ on the category of groups and any group $G$ the map 
$$\eta:G\to LG$$ 
is an NS-map. In particular, the cokernel of $\eta$ is abelian.}

\ 

In this section we prove that for a nilpotent group $N$ and any set of primes $P$ the map $N\to L_P N$ is an NS-map, where $L_P$ is Baumslag's $P$-localization.

The theorem of Nikolov--Segal has a stronger corollary: it implies that for any finitely generated group $G$ and any closed normal subgroup of its pro-finite completion $U\triangleleft \hat G_{\sf fin}$ the following holds 
$$[U,\hat G_{\sf fin}]=[U, \Im (G\to \hat G_{\sf fin})].$$ 
We prove an analogue of this statement for Baumslag's $P$-localization and nilpotent groups. Namely, we prove (Proposition \ref{prop_NS}) that for a nilpotent group $N$ and any $P$-divisible normal subgroup $U\triangleleft L_PN$ the following holds
$$ [U, L_PN]= [U, \Im ( N \to L_PN)].$$

We denote by $P$ a set of primes and by $L_P$ Baumslag's $P$-localization i.e. the localization with respect to the map $\ZZ\to \ZZ[P^{-1}]$. 
Since $L_P$ is right exact (Theorem \ref{th_Baumslag_l}), for any natural $c$ it takes nilpotent groups of class at most $c$ to nilpotent groups of class at most $c$ (Theorem  \ref{theorem_s.d._of_nilpotent}). Then, if we denote by ${\sf Nil}_c$ the category of nilpotent groups of class at most $c$, the restriction 
$$L_P^{(c)}:{\sf Nil}_c\to {\sf Nil}_c$$ 
is the left adjoint functor to the embedding of $P$-local nilpotent groups of class at most $c$. It follows that $L_PN $ coincides with $P'$-localization $N_{P'}$ considered by Hilton \cite[Def. in \S 4]{Hilton}, where $P'$ is the completion of $P.$ 
If $P$ is the set of all primes, this localization is known as Malcev's completion. 
In this section for the sake of simplicity we assume that $P$ is fixed and set 
$$L:=L_P.$$ 
 
If $H, K$ are (not necessarily normal) subgroups of $G,$ we define $[H,K]$ as a subgroup generated by commutators $[h,k],$ where $h\in H$ and $k\in K.$   
Moreover, we denote by $[H,_nK]$ the $n$-fold commutator subgroup by recursion: $[H,_0K]=H$ and $[H,_{n+1}K]:=[[H,_n K],K]$
$$[H,_n K]=[H,\underbrace{K, \dots,K}_n].$$
Note that the lower central series of a group $G$ can be defined as $\gamma_n G=[G,_{n-1}G].$ Moreover, if $U$ is a normal subgroup of $G$ the following holds
$$ \gamma_n(G\ltimes U)= \gamma_n G \ltimes [U,_{n-1} G].$$

% A homomorphism $f:H\to G$ is called {\it NS-map} (Nikolov--Segal map, see \cite{NikolovSegal}) if it satisfies 
% $$[G,G]=[G, \Im f].$$

\begin{Proposition}\label{prop_NS} Let $N$ be a nilpotent group. Then $\eta: N\to LN$ is an NS-map. Moreover, if $U$ is a normal $P$-local subgroup of $LN$, then 
$$[U,_n LN]=[U,_n \Im\: \eta]$$ 
for any $n\geq 1.$
\end{Proposition}
\begin{proof}
A classical construction of W. Magnus \cite{magnus1940gruppen}  associates to a group $G$ a graded Lie ring ${\sf gr}(G)$ of the lower central quotients 
$${\sf gr}(G):= \bigoplus_{n=1}^\infty \  \gamma_n G/\gamma_{n+1} G,$$
whose bracket is induced by the commutator in $G.$ 
In order to distinguish the bracket in ${\sf gr}(G)$ from the commutator in $G$ we will denote it by $[x,y]_{\sf gr}.$

Set $A=LN$ and $I=\Im\: \eta.$ Then we need to prove  that  
$[U,_n A] = [U,_n I]$ for $n\geq 1.$ Consider a graded Lie ring of the semidirect product
$${\sf gr}(A\ltimes U)=\bigoplus_{n\geq 1} \ \  \frac{\gamma_n A}{\gamma_{n+1} A}  \oplus \frac{[U,_{n-1}A]}{[U,_{n}A]}.$$
The summands $\frac{[U,_{n-1}A]}{[U,_{n}A]}$ form an ideal in this Lie ring. Hence, there is a well defined map
$$\frac{U}{[U,A]} \otimes A_{ab}^{\otimes n} \longrightarrow  \frac{[U,_{n}A]}{[U,_{n+1}A]}, \hspace{1cm} y\otimes x_1\otimes \dots \otimes x_n \mapsto [y,x_1,\dots,x_n]_{\sf gr}.$$

Consider some elements $X_1,\dots,X_n\in A$ and $Y\in U.$ Denote their images in $A_{ab}$ and in $U/[U,A]$ by $x_1,\dots,x_n\in A_{ab}$ and $y\in U/[U,A]$ respectively.  Since the map $I\to A$ is a $P'$-isomorphism (see \cite[\S 4]{Hilton}), there exists a natural $k_i$ whose prime factors are in $P$ such that $X_i^{k_i}\in I.$ Set $\bar X_i:=X_i^{k_i}\in I$ and denote their images in $A_{ab} $ by $\bar x_1,\dots,\bar x_{n}.$ Since $U$ is $P$-local, there exists $Z\in U$ such that $Z^{k_1\dots k_n}=Y.$ We denote its image in $U/[U,A]$ by $z.$ Then $\bar x_i=k_i x_i$ and $(k_1\dots k_n)z=y.$ Hence
$$y\otimes x_1\otimes \dots \otimes x_n = z \otimes \bar x_1 \otimes \dots \otimes \bar x_n.$$
Therefore, their images in $\frac{[U,_{n}A]}{[U,_{n+1}A]}$ are equal:
$$[y,x_1,\dots,x_n]_{\sf gr}=[z, \bar x_1,\dots, \bar x_n]_{\sf gr}.$$
It follows that any element of $[U,_nA]$ can be presented as an element of $[U,_nI]$ modulo $[U,_{n+1}A].$ Then  
\begin{equation}\label{eq_UnA}
[U,_nA]=[U,_nI][U,_{n+1}A]
\end{equation}
for $n\geq 1.$
Assume that $N$ is nilpotent of class $c.$ Hence $A$ is nilpotent of class at most $c.$ It follows that $[U,_c A]=1$ and 
$$[U,_{c-1} A]=[U,_{c-1} I].$$ 
Then by induction, using \eqref{eq_UnA}, we obtain $[U,_nA]=[U,_nI]$ for $n=c-1,c-2,\dots,1.$
\end{proof}

\end{document}